\documentclass[11pt]{amsart}

\usepackage[T1]{fontenc}
\usepackage{microtype}
\usepackage{amsmath,amssymb,mathtools}
\usepackage{amsthm}
\usepackage{booktabs}
\usepackage{array}
\usepackage{enumitem}
\usepackage[margin=1in]{geometry}
\usepackage[colorlinks=true,linkcolor=blue,citecolor=blue,urlcolor=blue]{hyperref}

\newtheorem{theorem}{Theorem}[section]
\newtheorem{proposition}[theorem]{Proposition}
\newtheorem{lemma}[theorem]{Lemma}
\newtheorem{corollary}[theorem]{Corollary}
\newtheorem{remark}[theorem]{Remark}
\newtheorem{definition}[theorem]{Definition}
\numberwithin{equation}{section}

\newcommand{\N}{\mathbb N}
\newcommand{\Z}{\mathbb Z}
\newcommand{\F}{\mathbb F}
\newcommand{\R}{\mathbb R}
\newcommand{\Pp}{\mathbb P}
\newcommand{\rad}{\operatorname{rad}}
\newcommand{\Vol}{\operatorname{Vol}}
\newcommand{\1}{\mathbf 1}

\title[Resolution of Erd\H{o}s Problem 1061]
{A resolution of Erd\H{o}s Problem 1061 on the sum-of-divisors function}
\author{Eric Li}
\date{June 24, 2026}
\thanks{Email addresses: \href{mailto:el593@cam.ac.uk}{el593@cam.ac.uk}, \href{mailto:contact@ericli.com}{contact@ericli.com}.}
\hypersetup{
  pdftitle={A resolution of Erdos Problem 1061 on the sum-of-divisors function},
  pdfauthor={Eric Li}
}

\makeatletter
\def\@setauthors{%
  \begingroup
  \def\thanks{\protect\thanks@warning}%
  \trivlist
  \centering\footnotesize \@topsep30\p@\relax
  \advance\@topsep by -\baselineskip
  \item\relax
  \author@andify\authors
  \def\\{\protect\linebreak}%
  \MakeUppercase{\authors}\par
  \vspace{0.75em}%
  {\normalfont\normalsize Trinity College, University of Cambridge\par}%
  \endtrivlist
  \endgroup
}
\makeatother

\begin{document}

\begin{abstract}
We resolve Erd\H{o}s Problem 1061, the question whether the number
\[
 S(x)=\#\{(a,b)\in\N^2:a+b\le x,
 \ \sigma(a)+\sigma(b)=\sigma(a+b)\}
\]
of ordered solutions has a linear asymptotic $S(x)\sim cx$.  In fact the
opposite extreme holds at every fixed logarithmic scale: for every $R>0$,
\[
 \lim_{x\to\infty}\frac{S(x)}{x(\log x)^R}=+\infty.
\]
The construction begins with three integers having the same abundancy index
and reduces the divisor-sum identity to two equations in six primes.  After a
linear change of variables, these equations lie on a split quadric.  A
three-parameter rational ruling of the quadric supplies many affine systems
of six linear forms.  An exact lattice-index calculation, an elementary
codimension-two parameter sieve, and Bienvenu's higher-dimensional
Siegel--Walfisz theorem give prime points uniformly on these planes.  Coprime
multiplier amplification then yields the stated resolution.
\end{abstract}

\maketitle

\section{Introduction}

Let $\sigma(n)=\sum_{e\mid n}e$.  Erd\H{o}s Problem 1061, a question of
Erd\H{o}s reported in Problem B15 of Guy's collection \cite{Guy}, asks for the
order of
\[
 S(x)=\#\{(a,b)\in\N^2:a+b\le x,
 \ \sigma(a)+\sigma(b)=\sigma(a+b)\},
\]
and in particular whether $S(x)\sim cx$ for some $c>0$.  Ordered pairs are
counted throughout.  The main result of this paper resolves the problem by
proving that $S(x)$ is larger than $x$ by every fixed power of $\log x$.

\begin{theorem}\label{thm:main}
For every fixed $R>0$,
\[
 \lim_{x\to\infty}\frac{S(x)}{x(\log x)^R}=+\infty.
\]
In particular, $S(x)\not\sim cx$ for every finite $c>0$.
\end{theorem}

\begin{remark}[Ordered and unordered conventions]\label{rem:ordered}
There are no diagonal solutions.  Indeed, if $a=b=m$ and
$m=2^v m_0$ with $m_0$ odd, then multiplicativity at the odd part gives
\[
 \frac{\sigma(2m)}{\sigma(m)}
 =\frac{\sigma(2^{v+1})}{\sigma(2^v)}
 =\frac{2^{v+2}-1}{2^{v+1}-1}>2.
\]
Hence $\sigma(2m)\ne2\sigma(m)$.  Every solution therefore occurs in a
pair $(a,b),(b,a)$ with $a\ne b$, so the corresponding unordered counting
function is exactly $S(x)/2$.  The same evenness is noted for the fixed-sum
counts in OEIS A110177 \cite{OEIS}.
\end{remark}

\newpage
\paragraph{\textbf{Use of artificial intelligence tools.}}
Large language models, primarily OpenAI's ChatGPT, were used extensively
throughout the research.  The author originated the ideas and research
directions, while the models contributed substantially to the technical
development of the work: they were used to explore and further develop the
ideas, produce technical lemmas and calculations, generate and debug code, and
assist in auditing the proof.  The author critically evaluated outputs from
different model instances, selected and synthesised promising elements, and
discarded or corrected unsuccessful or flawed suggestions.  All strategic
decisions were made by the author, who takes full responsibility for the
mathematical correctness of the final results.

\medskip
\noindent\textbf{Remark \ref*{rem:ordered} continued.}
The mechanism is quantitative.  A \emph{reduced core} is an ordered coprime
pair $(u,v)$.  We construct many cores which are not themselves solutions
but become solutions after multiplication by the fixed, non-coprime
multiplier $3600$.

\begin{theorem}[Two-height core theorem]\label{thm:cores}
There is an absolute constant $c_0$ with $0<c_0<1$ and the following property.  For every
$\kappa>0$, there exist constants
\[
 c_\kappa>0,\qquad Y_\kappa\ge3,\qquad Q_\kappa\ge1,
 \qquad Z_\kappa\ge100
\]
such that, whenever
\[
 Y\ge Y_\kappa,\qquad
 Q_\kappa\le Q\le(\log Y)^\kappa,\qquad
 Z_\kappa\le Z\le(\log Y)^\kappa,
\]
there are at least
\[
 c_\kappa\frac{YQ^2Z}{(\log Y)^6}
\]
distinct ordered pairs $(u,v)\in\N^2$ satisfying
\[
 (u,v)=1,\qquad c_0Y\le u+v\le Y,
\]
\[
 \sigma(3600u)+\sigma(3600v)=\sigma(3600(u+v)),
\]
but
\[
 \sigma(u)+\sigma(v)\ne\sigma(u+v).
\]
\end{theorem}

The only deep theorem invoked as a black box is the prime-supported
uniform theorem of Bienvenu \cite[Proposition~2.1]{Bienvenu}.  Its proof rests
on the Green--Tao linear-forms method and the subsequent
M\"obius--nilsequence and inverse-theorem machinery
\cite{GreenTao,GreenTaoMobius,GreenTaoZiegler}.  We state and establish the precise
specialization needed in Section~\ref{sec:primes}.  The broader
sum-of-divisors context is discussed, for example, by Pollack and Pomerance
\cite{PollackPomerance}; fixed-sum data for the present equation are
recorded in OEIS A110177 \cite{OEIS}.  Every lattice, local, and
parameter-sieve calculation used below is proved explicitly; the thresholds
inherited from the prime-pattern theorem enter only qualitatively.

For bibliographic orientation, the Erd\H{o}s Problems database currently lists
Problem 1061 as open and cautions that its literature coverage may be
incomplete \cite{Bloom}.  The proof below is self-contained apart from the
stated prime-pattern input.

Throughout,
\[
 \N=\{1,2,3,\ldots\},\qquad \Pp=\{2,3,5,\ldots\}.
\]
We write $\varphi$ for Euler's totient function, $\mu$ for the M\"obius
function, $\omega(n)$ for the number of distinct prime divisors of $n$, and
$\rad(n)$ for the product of those prime divisors.  The scale parameters $Y,Q,Z,T$ are positive real numbers;
$q,A,B,C$ and all lattice variables are integers.  Half-open intervals are
used when integer boxes are counted.  Replacing a real endpoint by its floor
or ceiling changes the relevant counts only by the boundary terms displayed
below.

The constants are chosen in the following order:
\[
 W\longrightarrow M_0\longrightarrow P\longrightarrow Z_0
 \longrightarrow q_0\longrightarrow Q_\kappa,Z_\kappa,Y_\kappa.
\]
Here $W$ is fixed in \eqref{eq:W}, $M_0$ comes from
Lemma~\ref{lem:denominators}, $P$ is the fixed small-prime cutoff in the
parameter sieve, and $Z_0,q_0$ are then chosen in that order.  All later
thresholds may depend on the fixed value of $\kappa$, but not on the varying
parameters within the stated ranges.

\section{A friendly coefficient triple}\label{sec:friendly}

Set
\[
 (K,M,N)=(54000,111600,1828800)=3600(15,31,508).
\]
Write
\[
 k=15,\qquad m=31,\qquad n=508,\qquad D=n-k-m=462.
\]
The factorizations
\[
 K=2^4 3^3 5^3,\qquad
 M=2^4 3^2 5^2\cdot31,\qquad
 N=2^6 3^2 5^2\cdot127
\]
give
\[
 \sigma(K)=193440,\qquad \sigma(M)=399776,
 \qquad \sigma(N)=6551168,
\]
and hence
\begin{equation}
 \frac{\sigma(K)}K=\frac{\sigma(M)}M=\frac{\sigma(N)}N
 =\frac{806}{225}.\label{eq:abundancy}
\end{equation}

Suppose that six pairwise distinct primes $p_1,\dots,p_6>127$ satisfy
\begin{align}
 kp_1p_2+mp_3p_4&=np_5p_6,\label{eq:product}\\
 k(p_1+p_2)+m(p_3+p_4)-n(p_5+p_6)&=D.\label{eq:linear}
\end{align}
Since $k+m-n=-D$, expansion of the shifted products gives
\begin{equation}
 k(p_1+1)(p_2+1)+m(p_3+1)(p_4+1)
 =n(p_5+1)(p_6+1).\label{eq:shifted}
\end{equation}
Consequently, with
\[
 a=Kp_1p_2,\qquad b=Mp_3p_4,
\]
we have $a+b=Np_5p_6$.  The six primes are distinct and avoid the prime
divisors of $KMN$, so multiplicativity is applicable to each displayed
factorization.  Combining \eqref{eq:abundancy} and \eqref{eq:shifted} gives
\[
 \sigma(a)+\sigma(b)=\sigma(a+b).
\]

\section{The split quadric and its ruling}\label{sec:quadric}

Put $x_i=p_i$.  Eliminate $x_6$ by \eqref{eq:linear}, and define
\[
 y_i=x_i-x_5\quad(1\le i\le4),\qquad y_5=x_5,
 \qquad \zeta=x_5+1.
\]
A direct substitution shows that \eqref{eq:product}--\eqref{eq:linear} are
equivalent to
\begin{equation}
 15y_1y_2+31y_3y_4+462y_5\zeta=0.\label{eq:split}
\end{equation}
Let
\[
 U=(15y_1,31y_3,462y_5),\qquad V=(y_2,y_4,\zeta).
\]
Then \eqref{eq:split} is $U\cdot V=0$.

For $\theta=(\alpha,\beta,\gamma)\in\R^3$, put
\[
 A_\theta=
 \begin{pmatrix}
 0&\alpha&\beta\\
 -\alpha&0&\gamma\\
 -\beta&-\gamma&0
 \end{pmatrix}.
\]
The graph $V=A_\theta U$ is isotropic because $A_\theta$ is skew-symmetric.
Intersect it with $\zeta-y_5=1$, and set
\[
 r=U_1=15y_1,\qquad s=U_2=31y_3.
\]
The resulting affine plane is parametrized by
\begin{align}
 x_1&=-1+(1/15-\beta)r-\gamma s,\label{eq:x1}\\
 x_2&=-1-462\beta+(-462\beta^2-\beta)r
 +(\alpha-462\beta\gamma-\gamma)s,\label{eq:x2}\\
 x_3&=-1-\beta r+(1/31-\gamma)s,\label{eq:x3}\\
 x_4&=-1-462\gamma+(-\alpha-462\beta\gamma-\beta)r
 +(-462\gamma^2-\gamma)s,\label{eq:x4}\\
 x_5&=-1-\beta r-\gamma s,\label{eq:x5}\\
 x_6&=\frac{15(x_1+x_2)+31(x_3+x_4)-508x_5-462}{508}.
 \label{eq:x6}
\end{align}
Substitution into $U\cdot V=0$ proves \eqref{eq:product}; the definition of
$x_6$ proves \eqref{eq:linear}.  Thus both equations hold identically on
every plane in this ruling.

\subsection{A two-height positive region}

Fix
\[
 \varepsilon=10^{-5}.
\]
For $Z\ge100$, impose
\begin{equation}
 Z^2\le\alpha\le2Z^2,
 \qquad Z\le\beta,\gamma\le2Z,\label{eq:parambox}
\end{equation}
and let
\begin{equation}
 \mathcal R_T=\{(r,s):-2T\le r\le-T,\ |s|\le\varepsilon|r|\}.
 \label{eq:RT}
\end{equation}

\begin{lemma}[Uniform archimedean bounds]\label{lem:arch}
There are absolute constants $0<c<C$ such that, for every $Z\ge100$ and
$T\ge10$, every point satisfying \eqref{eq:parambox} and
$(r,s)\in\mathcal R_T$ satisfies
\[
 cZT\le x_1,x_3,x_5\le CZT,
\]
\[
 cZ^2T\le x_2,x_4,x_6\le CZ^2T.
\]
Consequently,
\begin{equation}
 cZ^3T^2\le508x_5x_6\le CZ^3T^2.\label{eq:height}
\end{equation}
\end{lemma}

\begin{proof}
Write $\rho=-r\in[T,2T]$.  In each of $x_1,x_3,x_5$, the coefficient of $\rho$
after allowing the worst sign of $s$ is at least
\[
 Z-\frac1{15}-2\varepsilon Z>0.99Z.
\]
The corresponding upper bounds are immediate, and the constant $-1$ is
absorbed for $T\ge10$.

For $x_2$, the coefficient of $\rho$ after allowing the worst sign of $s$ is
at least
\[
 462Z^2+Z-\varepsilon(1850Z^2+2Z)>461Z^2.
\]
Its constant term has absolute value at most $1+924Z$, so
$x_2\asymp Z^2T$ uniformly.  Similarly, the coefficient of $\rho$ in $x_4$
is bounded below by
\[
 \alpha+462\beta\gamma+\beta
 -\varepsilon(462\gamma^2+\gamma)
 \ge 463Z^2+Z-\varepsilon(1848Z^2+2Z)>462Z^2,
\]
and its constant term has absolute value at most $1+924Z$.  Hence
$x_4\asymp Z^2T$ uniformly.  Finally, \eqref{eq:product}, which holds
identically on the plane, gives
\[
 x_6=\frac{15x_1x_2+31x_3x_4}{508x_5}\asymp Z^2T.
\]
All implied constants above are absolute, and \eqref{eq:height} follows.
\end{proof}

Let $\ell_i=(\partial_r x_i,\partial_s x_i)$ denote the linear part of
$x_i$.

\begin{lemma}[Finite complexity]\label{lem:complexity}
For $Z\ge100$, the six affine forms in
\eqref{eq:x1}--\eqref{eq:x6} have pairwise nonproportional linear parts.
\end{lemma}

\begin{proof}
The following table records the nonconstant numerator of each determinant
$\det(\ell_i,\ell_j)$ after removal of its displayed nonzero rational
constant.  The exact constants are listed in Appendix~\ref{app:algebra}.
\begin{center}
\begingroup
\small
\setlength{\tabcolsep}{4pt}
\begin{tabular}{c|c@{\quad}c|c}
\toprule
pair&determinant factor&pair&determinant factor\\
\midrule
$1,2$&$15\alpha\beta-\alpha+462\beta\gamma+\gamma$
&$1,3$&$15\beta+31\gamma-1$\\
$2,3$&$\beta(31\alpha-462\beta-1)$
&$1,4$&$\gamma(15\alpha+462\gamma+1)$\\
$2,4$&$\alpha(\alpha+\beta-\gamma)$
&$3,4$&$31\alpha\gamma-\alpha-462\beta\gamma-\beta$\\
$1,5$&$\gamma$&$2,5$&$\alpha\beta$\\
$3,5$&$\beta$&$4,5$&$\alpha\gamma$\\
$1,6$&$(15\alpha+462\gamma+1)(15\beta+31\gamma-1)$
&$2,6$&$(\alpha+\beta-\gamma)(31\alpha-462\beta-1)$\\
$3,6$&$(31\alpha-462\beta-1)(15\beta+31\gamma-1)$
&$4,6$&$(\alpha+\beta-\gamma)(15\alpha+462\gamma+1)$\\
$5,6$&$15\alpha\beta+31\alpha\gamma+\beta-\gamma$&&\\
\bottomrule
\end{tabular}
\endgroup
\end{center}
Every factor is nonzero in \eqref{eq:parambox}.  For example,
\[
 31\alpha-462\beta-1\ge31Z^2-924Z-1>0,
\]
\[
 \alpha+\beta-\gamma\ge Z^2-Z>0,
\]
and, minimizing successively in $\alpha,\beta,\gamma$,
\[
 31\alpha\gamma-\alpha-462\beta\gamma-\beta
 \ge31Z^3-925Z^2-2Z>0.
\]
The remaining factors are immediately positive from
\eqref{eq:parambox}.  Hence no determinant vanishes.  Since two affine
forms are affinely related precisely when their linear parts are
proportional, this is finite complexity in the sense of Bienvenu's
Definition~1.1.
\end{proof}

\section{Integral points on a rational ruling plane}\label{sec:lattice}

Let
\begin{equation}
 P_0=2\cdot3\cdot5\cdot7\cdot11\cdot31\cdot127,
 \qquad W=508P_0=4619990760.\label{eq:W}
\end{equation}
Take a squarefree integer $q\equiv1\pmod W$ and integers
$A,B,C\in W\Z$ satisfying
\begin{equation}
 qZ^2\le A<2qZ^2,
 \qquad qZ\le B,C<2qZ.\label{eq:ABCbox}
\end{equation}
Put
\[
 \alpha=A/q,\qquad\beta=B/q,\qquad\gamma=C/q.
\]

Write $t=x_5$ and consider the conditions
\begin{align}
 15By_1+31Cy_3+qt&=-q,\label{eq:Haff}\\
 31Ay_3+462Bt&\equiv0\pmod q,\label{eq:phi1}\\
 -15Ay_1+462Ct&\equiv0\pmod q.\label{eq:phi2}
\end{align}
We impose
\begin{equation}
 (ABC,q)=1.\label{eq:units}
\end{equation}

\begin{lemma}[Exact integral lift]\label{lem:integrallift}
An integer triple $(y_1,y_3,t)$ corresponds to an integral point of the
rational ruling plane if and only if \eqref{eq:Haff}--\eqref{eq:phi2}
hold and, after defining
\begin{equation}
 y_2=\frac{31Ay_3+462Bt}{q},
 \qquad
 y_4=\frac{-15Ay_1+462Ct}{q},\label{eq:y2y4lift}
\end{equation}
the integer
\begin{equation}
 \nu_6=15(x_1+x_2)+31(x_3+x_4)-508x_5-462,
 \label{eq:nu6lift}
\end{equation}
where
\begin{equation}
 x_i=y_i+t\quad(1\le i\le4),
 \qquad x_5=t,\label{eq:x15lift}
\end{equation}
is divisible by $508$.  In that case the lift is unique and
\begin{equation}
 x_6=\nu_6/508.\label{eq:x6lift}
\end{equation}
\end{lemma}

\begin{proof}
If an integral point lies on the ruling plane, the graph relation
$V=A_\theta U$ gives \eqref{eq:y2y4lift} and
\[
 t+1=-\frac{15By_1+31Cy_3}{q}.
\]
The latter identity is \eqref{eq:Haff}; integrality of $y_2,y_4$ gives
\eqref{eq:phi1}--\eqref{eq:phi2}; and integrality of $x_6$ is precisely
$508\mid\nu_6$.

Conversely, the two congruences make the quantities in
\eqref{eq:y2y4lift} integers, while \eqref{eq:Haff} gives the third graph
relation.  Thus, with
\[
 U=(15y_1,31y_3,462t),
 \qquad V=(y_2,y_4,t+1),
\]
we have $V=A_\theta U$.  Equations \eqref{eq:x15lift} therefore recover
$x_1,\ldots,x_5$ on the affine ruling plane, and the divisibility condition
makes \eqref{eq:x6lift} an integer satisfying the defining linear equation.
All coordinates are forced by the triple, so the lift is unique.
\end{proof}

\begin{lemma}[Exact projected index]\label{lem:index}
Suppose \eqref{eq:units} holds.  The set of projected integral points
$(y_1,y_3)$ is either empty or an affine translate of a rank-two sublattice
$\Gamma\subset\Z^2$ satisfying
\begin{equation}
 [\Z^2:\Gamma]=q^2\delta
 \qquad\text{for some }\delta\mid508.\label{eq:index}
\end{equation}
\end{lemma}

\begin{proof}
First omit the final congruence modulo $508$.  On the homogeneous kernel of
\eqref{eq:Haff}, the two congruences \eqref{eq:phi1} and
\eqref{eq:phi2} are equivalent, because
\[
 B(-15Ay_1+462Ct)-C(31Ay_3+462Bt)
 =-A(15By_1+31Cy_3).
\]
Set $b_1=15B$ and $b_2=31C$.  Since $(b_1,q)=1$, choose integers $c,k_0$
with
\[
 b_1c+b_2=qk_0.
\]
Then
\[
 v_1=(c,1,-k_0),\qquad v_2=(q,0,-b_1)
\]
form a basis of the primitive kernel of
$b_1y_1+b_2y_3+qt=0$, since
\[
 v_1\times v_2=-(b_1,b_2,q).
\]
Define the integer-valued linear form
\[
 \widetilde\phi(y_1,y_3,t)=31Ay_3+462Bt
\]
and let $\phi$ be its reduction modulo $q$.  We have
\[
 \widetilde\phi(v_2)=-6930B^2.
\]
Every prime divisor of $6930$ divides $W$, while $(B,q)=1$ and
$(q,W)=1$.  Thus this value is a unit modulo $q$, so $\phi$ is surjective
and $\ker\phi$ has exact index $q$ in the primitive hyperplane kernel.

Projection of the primitive hyperplane kernel to the $(y_1,y_3)$-plane has
image
\[
 \{(y_1,y_3):b_1y_1+b_2y_3\equiv0\pmod q\},
\]
which has exact index $q$ in $\Z^2$.  Projection is injective, because $t$
is then uniquely determined.  Hence the projected lattice before the $x_6$
condition has exact index $q^2$.  The condition that the numerator in
\eqref{eq:x6} be divisible by $508$ is one affine congruence on this
lattice.  If soluble, its direction kernel has index equal to the order of
a subgroup of $\Z/508\Z$, and therefore index $\delta$ for some $\delta\mid508$.
This proves \eqref{eq:index}.
\end{proof}

\begin{lemma}[Nonemptiness and a fixed local witness]\label{lem:witness}
Assume $(C,q)=1$.  The affine lattice in Lemma~\ref{lem:index} is nonempty.
Moreover, for every prime $p\mid W$ it contains a point at which all six
coordinates $x_i$ are nonzero modulo $p$.
\end{lemma}

\begin{proof}
Put
\[
 s_*=3918=362+508\cdot7.
\]
Choose $s$ by the Chinese remainder theorem so that
\begin{equation}
 s\equiv s_*\pmod W,
 \qquad 31Cs\equiv-1\pmod q.\label{eq:sCRT}
\end{equation}
Set
\[
 y_1=0,\qquad y_3=qs,\qquad t=-1-31Cs.
\]
The second congruence in \eqref{eq:sCRT} gives $q\mid t$.
Equations \eqref{eq:Haff}--\eqref{eq:phi2} hold, and the exact graph
relations give
\[
 y_2=31As+462B(t/q),\qquad y_4=462C(t/q),
\]
so $y_2,y_4$ are integers.  Since $A,B,C\in W\Z$,
$q\equiv1\pmod W$, $s\equiv3918\pmod W$, and $t\equiv-1\pmod W$,
we obtain
\[
 (x_1,x_2,x_3,x_4,x_5)\equiv(-1,-1,3917,-1,-1)\pmod W.
\]
Let $\nu$ be the numerator in \eqref{eq:x6}.  Then
\[
 \nu\equiv15(-2)+31(3917-1)-508(-1)-462
 =508\cdot239\pmod W.
\]
As $W=508P_0$, it follows that
\[
 x_6=239+P_0m
\]
for some integer $m$.  Every prime divisor of $W$ also divides $P_0$, so
for every $p\mid W$,
\begin{equation}
 (x_1,\dots,x_6)\equiv(-1,-1,3917,-1,-1,239)\pmod p.
 \label{eq:smallwitness}
\end{equation}
Indeed,
\[
 \gcd(3917\cdot239,W)=1,
\]
so each displayed coordinate is a unit modulo every prime divisor of $W$.
\end{proof}

\begin{lemma}[Uniform lattice coordinates]\label{lem:coordinates}
Let the affine lattice be nonempty, and put
$\Delta=[\Z^2:\Gamma]$.  There are absolute constants $C_0,C_1$ such that,
if $T\ge C_0\Delta$, the affine lattice admits a bijective parametrization
by $n\in\Z^2$ with the following properties:
\begin{enumerate}[label=\textup{(\roman*)}]
\item the pullback $\mathcal K$ of $\mathcal R_T$ is a convex body contained
in $[-C_1T,C_1T]^2$;
\item
\begin{equation}
 \Vol(\mathcal K)=\frac{\operatorname{area}(\mathcal R_T)}
 {15\cdot31\cdot\Delta}
 \asymp \frac{T^2}{\Delta}
 \asymp\frac{T^2}{q^2};\label{eq:Kvolume}
\end{equation}
\item each $x_i$ becomes an integer affine-linear form in $n$, with every
linear coefficient and constant term bounded by
\begin{equation}
 O(Z^2\Delta)=O(Z^2q^2).\label{eq:coeffbound}
\end{equation}
\end{enumerate}
The determinant of the linear map $n\mapsto(r,s)$ has absolute value
$15\cdot31\cdot\Delta$.  All implied constants are absolute.
\end{lemma}

\begin{proof}
Let $\widetilde\Gamma$ denote the full homogeneous direction lattice,
written in the coordinate space
\[
 (y_1,y_2,y_3,y_4,t,x_6),
\]
after every integrality congruence has been imposed.  Projection to
$(y_1,y_3)$ is a bijection
\[
 \pi:\widetilde\Gamma\longrightarrow\Gamma.
\]
Indeed, the homogeneous version of \eqref{eq:Haff} determines $t$ uniquely
from $(y_1,y_3)$; the two exact graph equations then determine $y_2,y_4$;
and $x_6$ is determined by \eqref{eq:x6}.  Existence of a lift is part of
the definition of $\Gamma$, and these uniqueness statements show that it is
unique.

Choose a Gauss-reduced basis $g_1,g_2$ of $\Gamma$.  Thus
\[
 |g_1|\le |g_2|,\qquad
 |\det(g_1,g_2)|=\Delta,
 \qquad |\sin\angle(g_1,g_2)|\ge\frac{\sqrt3}{2}.
\]
Since $g_1$ is a nonzero integer vector, $|g_1|\ge1$, and hence
\[
 |g_2|
 \le \frac{2\Delta}{\sqrt3\,|g_1|}
 \le\frac{2\Delta}{\sqrt3}.
\]
If $G$ is the matrix with columns $g_1,g_2$, then its largest singular value
is $O(\Delta)$, whereas
\[
 \sigma_{\min}(G)=\frac{|\det G|}{\sigma_{\max}(G)}\gg1.
\]
Let $\widetilde g_i$ be the unique lift of $g_i$ to
$\widetilde\Gamma$.  Choose a point $z_0$ of the affine lattice whose
projection lies in the centered fundamental parallelogram
\[
 \left\{t_1g_1+t_2g_2:-\tfrac12\le t_1,t_2<\tfrac12\right\}.
\]
Then its projected norm is at most $(|g_1|+|g_2|)/2=O(\Delta)$.  The map
\[
 n=(n_1,n_2)\longmapsto z_0+n_1\widetilde g_1+n_2\widetilde g_2
\]
is a bijection from $\Z^2$ to the full affine lattice.

The region $\mathcal R_T$ is expressed in
\[
 (r,s)=(15y_1,31y_3).
\]
The linear part of the map $n\mapsto(r,s)$ is
$\operatorname{diag}(15,31)G$, whose determinant has absolute value
$15\cdot31\cdot\Delta$.  The singular-value bounds above and the estimate
for the affine representative imply, when $T\ge C_0\Delta$, that the
inverse image $\mathcal K$ is contained in $[-C_1T,C_1T]^2$.  It is convex
because $\mathcal R_T$ is convex.  Change of variables gives the exact area
identity in \eqref{eq:Kvolume}.

Writing $\alpha=A/q$, $\beta=B/q$, and $\gamma=C/q$, the exact
graph equations, now expanded as affine functions of $(y_1,y_3)$, are
\begin{align}
 t&=-1-15\beta y_1-31\gamma y_3,\label{eq:texpanded}\\
 y_2&=-462\beta-6930\beta^2y_1
 +(31\alpha-14322\beta\gamma)y_3,\label{eq:y2expanded}\\
 y_4&=-462\gamma+(-15\alpha-6930\beta\gamma)y_1
 -14322\gamma^2y_3.\label{eq:y4expanded}
\end{align}
Since $x_i=y_i+t$ for $1\le i\le4$, $x_5=t$, and $x_6$ is the fixed
linear combination in \eqref{eq:x6}, every coefficient and every constant
term of
\[
 t,\ y_2,\ y_4,\ x_1,\ldots,x_6
\]
as an affine function of $(y_1,y_3)$ is $O(Z^2)$.  The projected basis
vectors and the centered affine representative have norm $O(\Delta)$, so
composition with the preceding parametrization gives
\eqref{eq:coeffbound}, for both the linear coefficients and the affine
constants.  These pulled-back coefficients are integers: the value at
$n=0$ is integral, and the coefficient in the $e_j$ direction is the
difference of the same integer coordinate at two points of the full
integral lattice.
\end{proof}

\section{Many good denominators and ruling planes}\label{sec:sieve}

For every positive integer $q$ coprime to $W$, put
\[
 H(q)=\sum_{p\mid q}\frac1p.
\]

\begin{lemma}[Positive-density denominators]\label{lem:denominators}
There are constants $M_0,c_W>0$ and $Q_0\ge1$ such that, for every
$Q\ge Q_0$,
\begin{equation}
 \#\{q\in[Q,2Q]:q\equiv1\pmod W,\ \mu^2(q)=1,\ H(q)\le M_0\}
 \ge c_WQ.\label{eq:goodq}
\end{equation}
\end{lemma}

\begin{proof}
Using $\mu^2(q)=\sum_{e^2\mid q}\mu(e)$ and observing that
$q\equiv1\pmod W$ forces $(q,W)=1$, the Chinese remainder theorem gives
\begin{align*}
 &\#\{q\in[Q,2Q]:q\equiv1\pmod W,\ \mu^2(q)=1\}\\
 &\quad=\sum_{\substack{e\le\sqrt{2Q}\\(e,W)=1}}\mu(e)
 \left(\frac{Q}{We^2}+O(1)\right).
\end{align*}
The main term is
\begin{align*}
 \frac QW\sum_{\substack{e\le\sqrt{2Q}\\(e,W)=1}}
 \frac{\mu(e)}{e^2}
 &=\frac QW\prod_{p\nmid W}(1-p^{-2})
 +O_W\left(Q\sum_{e>\sqrt{2Q}}\frac1{e^2}\right)\\
 &=\frac QW\prod_{p\nmid W}(1-p^{-2})+O_W(Q^{1/2}),
\end{align*}
and the accumulated $O(1)$ errors also contribute $O(Q^{1/2})$.
Therefore
\begin{equation}
 \begin{aligned}
 &\#\{q\in[Q,2Q]:q\equiv1\pmod W,\ \mu^2(q)=1\}\\
 &\qquad=\frac QW\prod_{p\nmid W}(1-p^{-2})+O_W(Q^{1/2}).
 \end{aligned}
 \label{eq:squarefreeprogression}
\end{equation}
Also there is a constant $C_H=C_H(W)>0$ such that, for all sufficiently
large $Q$,
\[
 \sum_{\substack{Q\le q\le2Q\\q\equiv1\ (W)}}H(q)
 \le\sum_{\substack{p\le2Q\\p\nmid W}}
 \frac1p\left(\frac{Q}{Wp}+1\right)
 \le C_HQ.
\]
By \eqref{eq:squarefreeprogression}, there is a constant
$c_{\rm sf}=c_{\rm sf}(W)>0$ such that the squarefree progression contains
at least $c_{\rm sf}Q$ integers for all sufficiently large $Q$.  Choose
\[
 M_0=\frac{4C_H}{c_{\rm sf}}.
\]
Since $H(q)\ge0$, Markov's inequality gives
\[
 \#\{q\in[Q,2Q]:q\equiv1\pmod W,\ H(q)>M_0\}
 \le\frac{C_HQ}{M_0}=\frac{c_{\rm sf}Q}{4}.
\]
This upper bound remains valid after restricting to squarefree $q$.
Consequently at least $3c_{\rm sf}Q/4$ squarefree denominators satisfy
$H(q)\le M_0$.  Taking $c_W=c_{\rm sf}/2$ and enlarging $Q_0$ proves
\eqref{eq:goodq}.
\end{proof}

For $p\mid q$, impose
\begin{equation}
 ABC(462C+15A)(31A-462B)(15B+31C)\not\equiv0\pmod p.
 \label{eq:qconditions}
\end{equation}
For $p\nmid Wq$, define the three bad systems
\begin{align}
 A&\equiv0, &462B+q&\equiv0 \pmod p,\label{eq:bad1}\\
 A&\equiv0, &462C+q&\equiv0 \pmod p,\label{eq:bad2}\\
 31A-462B-q&\equiv0,
 &15A+462C+q&\equiv0 \pmod p.\label{eq:bad3}
\end{align}

\begin{lemma}[Complete residue blocks]\label{lem:boxcount}
Let $I_1,I_2,I_3$ be finite intervals of consecutive integers, with
$L_i=|I_i|$, and let
\[
 \mathcal B=I_1\times I_2\times I_3.
\]
For an integer $m\ge1$ and a set
$\Omega\subseteq(\Z/m\Z)^3$,
\begin{equation}
 \#\{x\in\mathcal B:x\bmod m\in\Omega\}
 \ge |\Omega|\prod_{i=1}^3
 \max\left\{\left\lfloor\frac{L_i}{m}\right\rfloor-1,0\right\}.
 \label{eq:blocklower}
\end{equation}
In particular, if $L_i\ge4m$ for all $i$, then
\begin{equation}
 \#\{x\in\mathcal B:x\bmod m\in\Omega\}
 \ge \frac{|\Omega|}{8m^3}L_1L_2L_3.
 \label{eq:blocklower2}
\end{equation}
\end{lemma}

\begin{proof}
In any interval of $L_i$ consecutive integers there are at least
\[
 \max\{\lfloor L_i/m\rfloor-1,0\}
\]
disjoint complete blocks of $m$ consecutive integers whose initial points
are multiples of $m$.  Every
Cartesian product of three such blocks contains each residue triple modulo
$m$ exactly once, proving \eqref{eq:blocklower}.  If $L_i/m\ge4$, then
$\lfloor L_i/m\rfloor-1\ge L_i/(2m)$, which gives
\eqref{eq:blocklower2}.
\end{proof}

\begin{proposition}[Quantitative parameter sieve]\label{prop:parametersieve}
There are absolute constants $P>127$, $Z_0\ge100$, $q_0\ge1$, and
$c_2>0$ with the following property.  Suppose that $q$ is squarefree,
\[
 q\equiv1\pmod W,\qquad H(q)\le M_0,\qquad q\ge q_0,
\]
and $Z\ge Z_0$.  Then the box \eqref{eq:ABCbox} contains at least
\begin{equation}
 c_2q^3Z^4\label{eq:planecount}
\end{equation}
triples $(A,B,C)\in(W\Z)^3$ satisfying \eqref{eq:qconditions} for every
$p\mid q$ and avoiding \eqref{eq:bad1}--\eqref{eq:bad3} for every
$p\nmid Wq$.
\end{proposition}

\begin{proof}
Write $A=Wa$, $B=Wb$, and $C=Wc$.  The three intervals for $(a,b,c)$ have
cardinalities
\begin{equation}
 L_A\asymp_W qZ^2,
 \qquad L_B\asymp_W qZ,
 \qquad L_C\asymp_W qZ,
 \label{eq:scaledlengths}
\end{equation}
with absolute endpoint errors at most $2$.  Since $(W,q)=1$, all congruence
conditions may be transferred between $(A,B,C)$ and $(a,b,c)$ without
changing their ranks or densities.

For a prime $p\mid q$, the forbidden set in \eqref{eq:qconditions} is a
union of six affine hyperplanes in $\F_p^3$.  Hence at least
$p^3-6p^2$ residue triples survive.  Every such $p$ is at least $13$.  Since $6/p<1/2$ and
$\log(1-x)\ge-2x$ for $0\le x\le1/2$, we have
\begin{equation}
 \prod_{p\mid q}(1-6/p)
 \ge \exp\left(-12\sum_{p\mid q}\frac1p\right)
 \ge e^{-12M_0}.
 \label{eq:q-density}
\end{equation}

We next impose the bad-system exclusions at a fixed set of small primes.
For the moment let $P>127$ be fixed, and define
\[
 M_P=\prod_{\substack{p\le P\\p\nmid Wq}}p,
 \qquad
 M_*=\prod_{\substack{p\le P\\p\nmid W}}p.
\]
Then $M_P\mid M_*$ and $(q,M_P)=1$.  Each system
\eqref{eq:bad1}--\eqref{eq:bad3} has rank two over $\F_p$ whenever
$p\nmid Wq$.  For \eqref{eq:bad1} and \eqref{eq:bad2}, the two coefficient
rows are respectively
\[
 (1,0,0),(0,462,0)
 \quad\hbox{and}\quad
 (1,0,0),(0,0,462),
\]
which are independent because $p\nmid462$.  For \eqref{eq:bad3}, the rows
\[
 (31,-462,0),\qquad(15,0,462)
\]
are independent because $p\nmid31\cdot462$.  Thus each bad system contains
exactly $p$ of the $p^3$ residue triples, and their union contains at most
$3p$.  By the Chinese remainder theorem, the set $\Omega_P$ of residue
triples modulo $qM_P$ surviving both the denominator-prime conditions and
all small-prime bad-system exclusions has density
\begin{align}
 \frac{|\Omega_P|}{(qM_P)^3}
 &\ge \prod_{p\mid q}(1-6/p)
 \prod_{\substack{p\le P\\p\nmid Wq}}(1-3/p^2)\notag\\
 &\ge e^{-12M_0}
 \prod_{p\nmid W}(1-3/p^2)
 =:\vartheta>0.\label{eq:small-density}
\end{align}
The infinite product is positive, since all primes outside $W$ exceed $11$
and $\sum_p p^{-2}<\infty$.

Choose
\begin{equation}
 Z_0\ge\max\{100,10WM_*\}.
 \label{eq:Z0choice}
\end{equation}
For $Z\ge Z_0$, the endpoint errors in \eqref{eq:scaledlengths}
give, for example,
\[
 L_B,L_C\ge qZ/W-1\ge4qM_P,
 \qquad L_A\ge qZ^2/W-1\ge4qM_P.
\]
Thus all three sides contain at least four complete blocks modulo $qM_P$,
uniformly in $q$.  Lemma~\ref{lem:boxcount}, applied with $m=qM_P$, now gives
at least
\begin{equation}
 \frac{\vartheta}{8}L_AL_BL_C
 \ge \vartheta_0q^3Z^4
 \label{eq:smallsurvivors}
\end{equation}
triples surviving every condition at primes $p\le P$ and at primes
$p\mid q$, where $\vartheta_0>0$ depends only on the fixed $W$ and $M_0$.
This is a complete-block lower bound; no unquantified CRT boundary term
remains.

It remains to discard triples for which one of the bad systems occurs at a
prime $p>P$.  For this upper bound we discard the restrictions
$A,B,C\in W\Z$ and count all integer triples in the surrounding intervals;
this can only increase the number of bad triples.  The resulting side lengths
satisfy
\[
 L_A'\asymp qZ^2,
 \qquad L_B',L_C'\asymp qZ.
\]
For \eqref{eq:bad1}, the number of triples associated with a fixed prime is
at most
\begin{align*}
 (L_A'/p+1)(L_B'/p+1)L_C'
 &\le \frac{L_A'L_B'L_C'}{p^2}
 +\frac{L_A'L_C'+L_B'L_C'}p+L_C'\\
 &\ll \frac{q^3Z^4}{p^2}
      +\frac{q^2Z^3}{p}+qZ^2.
\end{align*}
For \eqref{eq:bad2} one analogously obtains
\begin{align*}
 (L_A'/p+1)L_B'(L_C'/p+1)
 &\le \frac{L_A'L_B'L_C'}{p^2}
 +\frac{L_A'L_B'+L_B'L_C'}p+L_B'\\
 &\ll \frac{q^3Z^4}{p^2}
      +\frac{q^2Z^3}{p}+qZ^2.
\end{align*}
For \eqref{eq:bad3}, choose $A$ first.  The two equations then prescribe
one residue class of $B$ and one residue class of $C$, so the count is at
most
\begin{align*}
 L_A'(L_B'/p+1)(L_C'/p+1)
 &\le \frac{L_A'L_B'L_C'}{p^2}
 +\frac{L_A'(L_B'+L_C')}p+L_A'\\
 &\ll \frac{q^3Z^4}{p^2}
      +\frac{q^2Z^3}{p}+qZ^2.
\end{align*}

Only primes $p\ll qZ^2$ can occur.  Indeed, each bad pair contains a nonzero
integer of size $O(qZ^2)$ that must be divisible by $p$.  For
\eqref{eq:bad1} and \eqref{eq:bad2} this integer is $A>0$.  For
\eqref{eq:bad3} it is the first expression, which is strictly positive for
$Z\ge100$:
\[
 31A-462B-q\ge q(31Z^2-924Z-1)>0.
\]
Summing over all three systems, using
\[
 \sum_{p>P}p^{-2}\ll P^{-1},\qquad
 \sum_{p\le X}p^{-1}\ll\log X,
 \qquad \pi(X)\le X,
\]
and noting that $\log(qZ^2)\le2\log(qZ)$ for $q,Z\ge2$, shows that, if
$E_{\rm tail}$ denotes the number discarded at primes $p>P$,
then
\begin{equation}
 E_{\rm tail}\ll
 \frac{q^3Z^4}{P}+q^2Z^3\log(qZ)+q^2Z^4.
 \label{eq:tailall}
\end{equation}
After division by $q^3Z^4$, \eqref{eq:tailall} becomes
\begin{equation}
 \frac{E_{\rm tail}}{q^3Z^4}
 \le C\left(\frac1P+\frac{\log(qZ)}{qZ}+\frac1q\right)
 \label{eq:tailnormalized}
\end{equation}
for an absolute constant $C$.  We now fix the constants in the announced
order.  First choose $P>127$ so large that the first term in
\eqref{eq:tailnormalized} is at most $\vartheta_0/4$.  Next choose $Z_0$ as in
\eqref{eq:Z0choice}, using this value of $P$.  Finally choose $q_0$ so large
that $q_0Z_0>\mathrm e$ and, for all $q\ge q_0$ and $Z\ge Z_0$,
\[
 C\left(\frac{\log(qZ)}{qZ}+\frac1q\right)
 \le\frac{\vartheta_0}{4}.
\]
This is uniform because $u\mapsto(\log u)/u$ is decreasing for $u>\mathrm e$.
Subtracting the large-prime tail from \eqref{eq:smallsurvivors} leaves at
least $(\vartheta_0/2)q^3Z^4$ triples.  This proves the proposition with
$c_2=\vartheta_0/2$.
\end{proof}

\begin{definition}\label{def:retained}
A \emph{retained denominator} at scale $Q$ is an integer $q\in[Q,2Q]$
counted by Lemma~\ref{lem:denominators} and satisfying $q\ge q_0$.  A
\emph{retained parameter triple} over such a denominator is a triple
$(A,B,C)$ counted by Proposition~\ref{prop:parametersieve}.  A
\emph{retained plane} is the ruling plane associated with a retained
denominator and a retained parameter triple.
\end{definition}

\begin{lemma}[Distinct retained planes]\label{lem:distinctplanes}
Distinct retained quadruples determine distinct affine ruling planes.
\end{lemma}

\begin{proof}
The conditions at every prime divisor of $q$ give
$(A,q)=(B,q)=(C,q)=1$.  Thus $q$ is the reduced denominator of each of the
three fractions $A/q,B/q,C/q$.  Equality of two rational parameter triples
therefore forces equality of their denominators and then of all three
numerators.  Hence distinct retained quadruples give distinct skew matrices
and distinct homogeneous graph subspaces.

It remains to pass from graph subspaces to their affine slices.  The difference
space of such a slice is the two-dimensional kernel of $\zeta-y_5$ inside the
graph subspace.  Adjoining any point of the slice, which has $\zeta-y_5=1$,
recovers the whole three-dimensional graph subspace.  Thus equality of two
affine slices would imply equality of their graph subspaces and hence of their
skew matrices.  The retained planes are therefore distinct.
\end{proof}

\section{Local admissibility}\label{sec:local}

Fix a retained plane and a lattice-coordinate system from
Lemma~\ref{lem:coordinates}.  Let
$\Psi=(\psi_1,\dots,\psi_6):\Z^2\to\Z^6$ be the six integer affine forms
obtained from $x_1,\dots,x_6$.  For a prime $p$, define
\begin{equation}
 \Lambda_p(m)=\frac{p}{p-1}\1_{p\nmid m},
 \qquad
 \beta_p=p^{-2}\sum_{n\in\F_p^2}\prod_{i=1}^6\Lambda_p(\psi_i(n)).
 \label{eq:localfactor}
\end{equation}
We use Bienvenu's terminology: a system is \emph{admissible} if it has
finite complexity and $\beta_p>0$ for every prime $p$.  Finite complexity
is preserved by the invertible affine reparametrization of the lattice, so
Lemma~\ref{lem:complexity} applies to $\Psi$.

\subsection{Primes away from \texorpdfstring{$Wq$}{Wq}}

If $p\nmid Wq$, then $p\nmid\Delta=q^2\delta$.  Applying
Lemma~\ref{lem:sublatticereduction} to $\Gamma\subseteq\Z^2$ shows that the
reduction of the projected lattice is all of $\F_p^2$.  The lattice-coordinate
map is therefore an affine isomorphism modulo $p$ onto the full projected
plane; multiplication by the units $15$ and $31$ then identifies it with
the full $(r,s)$-plane modulo $p$.

\begin{lemma}[Zero-form criteria]\label{lem:zeroforms}
For $p\nmid Wq$, the forms $x_1,x_3,x_5$ are never zero affine forms modulo
$p$.  The remaining forms are zero affine forms precisely in the following
cases:
\begin{align*}
 x_2\equiv0&\iff A\equiv0,\quad462B+q\equiv0\pmod p,\\
 x_4\equiv0&\iff A\equiv0,\quad462C+q\equiv0\pmod p,\\
 x_6\equiv0&\iff31A-462B-q\equiv0,\quad
 15A+462C+q\equiv0\pmod p.
\end{align*}
\end{lemma}

\begin{proof}
The constants of $x_1,x_3,x_5$ are $-1$.  For $x_2$, its constant
vanishes exactly when $462\beta+1=0$; under that condition its
$r$-coefficient vanishes automatically and its $s$-coefficient equals
$\alpha$.  This proves the first criterion.  The argument for $x_4$ is
identical.

For $x_6$, put
\[
 E=462(15\beta+31\gamma)+46.
\]
Direct expansion of \eqref{eq:x6} gives
\[
 508x_6(0,0)=-E,
\]
\[
 508\,\partial_rx_6+\beta E=-(31\alpha-462\beta-1),
\]
\[
 508\,\partial_sx_6+\gamma E=15\alpha+462\gamma+1.
\]
Thus all three affine coefficients vanish exactly under the two congruences
in the statement; conversely, those two congruences imply $E=0$.
\end{proof}

The parameter sieve excludes all three zero-form possibilities.

\subsection{Primes dividing the denominator}

Let $p\mid q$.  Because $q$ is squarefree and $q\equiv1\pmod W$, the
constants $15,31,462,508$ are units modulo $p$.  Reducing
\eqref{eq:Haff}--\eqref{eq:phi2} gives
\begin{equation}
 U=\lambda(C,-B,A).\label{eq:Uspecial}
\end{equation}
The exact graph equations also give
\begin{equation}
 Cy_2-By_4+A(t+1)=0.\label{eq:Vspecial}
\end{equation}
Taking $w=y_2$, the restrictions of the first five coordinates to this
special plane are
\begin{align*}
 x_1&=\lambda\left(\frac C{15}+\frac A{462}\right),\\
 x_2&=w+\frac{A\lambda}{462},\\
 x_3&=\lambda\left(-\frac B{31}+\frac A{462}\right),\\
 x_4&=\frac{Cw+A(1+A\lambda/462)}B+\frac{A\lambda}{462},\\
 x_5&=\frac{A\lambda}{462},
\end{align*}
and the coefficient of $w$ in $x_6$ is
\begin{equation}
 \frac{15B+31C}{508B}.\label{eq:x6w}
\end{equation}
The denominator-prime conditions certify nonvanishing coordinate by
coordinate:
\begin{center}
\begin{tabular}{c|c}
\toprule
coordinate&nonzero affine coefficient modulo $p$\\
\midrule
$x_1$&$(462C+15A)/(15\cdot462)$ in the $\lambda$ direction\\
$x_2$&$1$ in the $w$ direction\\
$x_3$&$(31A-462B)/(31\cdot462)$ in the $\lambda$ direction\\
$x_4$&$C/B$ in the $w$ direction\\
$x_5$&$A/462$ in the $\lambda$ direction\\
$x_6$&$(15B+31C)/(508B)$ in the $w$ direction\\
\bottomrule
\end{tabular}
\end{center}
Every denominator in this table is a unit modulo $p$, and every displayed
numerator is nonzero by \eqref{eq:qconditions}.  Thus no coordinate is the
zero affine form on the special plane.

\begin{lemma}[Reduction of a prime-to-$p$ sublattice]\label{lem:sublatticereduction}
Let $L'\subseteq L$ be rank-$r$ lattices of finite index $\iota=[L:L']$.
If $p\nmid\iota$, then inclusion induces an isomorphism
\[
 L'/pL'\longrightarrow L/pL.
\]
More generally, if a nonempty affine coset $z'+L'$ is contained in a
nonempty affine coset $z+L$, then the two affine cosets have the same
reduction modulo $p$.
\end{lemma}

\begin{proof}
Since $\iota L\subseteq L'$, choose $a,b\in\Z$ with $a\iota=1+bp$.  For any
$x\in L$, the element $a\iota x\in L'$ is congruent to $x$ modulo $pL$, proving
surjectivity.  If $x\in L'$ and $x=py$ with $y\in L$, then
\[
 y=a\iota y-bpy=a\iota y-bx\in L',
\]
so $x\in pL'$, proving injectivity.  For the affine statement, choose a common point $z_c\in z'+L'$.  Since
$z'+L'\subseteq z+L$, one has $z'+L'=z_c+L'$ and $z+L=z_c+L$.
Translating the lattice isomorphism by the reduction of $z_c$ proves equality
of the two affine reductions.
\end{proof}

\begin{lemma}[Full special fibre]\label{lem:specialfibre}
Let
\[
 \mathcal A_0=z_0+L_0
\]
be the affine lattice of integer triples $(y_1,y_3,t)$ satisfying
\eqref{eq:Haff}--\eqref{eq:phi2}, before the final $x_6$-condition, and let
\[
 \mathcal A=z_*+L\subseteq\mathcal A_0
\]
be the full affine lattice after that condition is imposed.  Then
$L\subseteq L_0$ and
\[
 [L_0:L]=\delta\qquad\text{for some }\delta\mid508.
\]
For every prime $p\mid q$, the reduction modulo $p$ of $\mathcal A$, after
the exact lift of Lemma~\ref{lem:integrallift}, is the entire affine plane
described by \eqref{eq:Uspecial}--\eqref{eq:Vspecial}.
\end{lemma}

\begin{proof}
The index assertion is the last step of Lemma~\ref{lem:index}.  The
formulas in Lemma~\ref{lem:integrallift} define on $\mathcal A_0$ an
affine coordinate map
\[
 \iota_0:\mathcal A_0\longrightarrow\Z[1/508]^6;
\]
its restriction to $\mathcal A$ takes values in $\Z^6$.  Since $p\mid q$
and $(q,508)=1$, reduction of $\iota_0$ modulo $p$ is well defined.  We
first work with $\mathcal A_0$.  Use the basis $v_1,v_2$ of the primitive
homogeneous hyperplane kernel from Lemma~\ref{lem:index}, and retain the
integer-valued form $\widetilde\phi$.  Put
\[
 a_0=\widetilde\phi(v_1),\qquad
 u_0=\widetilde\phi(v_2)=-6930B^2.
\]
The integer $u_0$ is a unit modulo $q$.  Choose $z\in\Z$ with
$a_0+zu_0\equiv0\pmod q$, and set
\[
 w_1=v_1+zv_2,
 \qquad w_2=qv_2.
\]
Both vectors lie in $L_0$.  Relative to $(v_1,v_2)$ their coordinate matrix
is
\[
 \begin{pmatrix}1&0\\ z&q\end{pmatrix},
\]
whose determinant is $q$.  Since $L_0$ has exact index $q$ in the primitive
hyperplane kernel, $(w_1,w_2)$ is a basis of $L_0$.

Apply the linear part $D\iota_0$ of this affine coordinate map and use
$(\lambda,w)$ as coordinates on the direction space of the special plane,
where $U=\lambda(C,-B,A)$ and $w=y_2$.  Modulo $p$, the image of $w_1$
has $y_3=1$, so $-B\lambda=31$.  The vector $w_2$ has $U\equiv0$, while
its lifted change in $y_2$ is
\[
 \frac{\widetilde\phi(qv_2)}q
 =\widetilde\phi(v_2)=-6930B^2.
\]
Thus the two reduced direction vectors have the triangular coordinate
matrix
\begin{equation}
 D\iota_0(w_1)\longmapsto
 \left(-\frac{31}{B},\,*\right),
 \qquad
 D\iota_0(w_2)\longmapsto
 \left(0,-6930B^2\right).
 \label{eq:specialbasiscoords}
\end{equation}
Both diagonal entries are units modulo $p$: $B$ is a unit by
\eqref{eq:qconditions}, and every prime divisor of $6930$ divides $W$.
Hence the determinant of this coordinate matrix is nonzero modulo $p$, and
the reductions of the lifted vectors $w_1,w_2$ form a basis of the complete
two-dimensional direction space of the special plane.

Now $z_*\in\mathcal A\subseteq\mathcal A_0$.  Reducing the exact affine
relations at its lift gives
\[
 U(z_*)=\lambda_*(C,-B,A),
 \qquad
 C y_2(z_*)-B y_4(z_*)+A(t(z_*)+1)=0
\]
for some $\lambda_*\in\F_p$.  Hence the reduction of $z_*+L_0$ is the
entire special affine plane.  Since $z_*\in z_0+L_0$, the precondition
coset may indeed be written with the common base point as
$\mathcal A_0=z_*+L_0$.

Finally, $p\nmid\delta$ because $p\mid q$, $(q,508)=1$, and $\delta\mid508$.
Lemma~\ref{lem:sublatticereduction} gives an isomorphism
$L/pL\to L_0/pL_0$ and equality of the reductions of the common-base-point
cosets $z_*+L$ and $z_*+L_0$.  The direction map
$D\iota_0:L_0\to\Z[1/508]^6$ is a homomorphism, and $508$ is a unit
modulo $p$, so congruent lattice directions have congruent coordinate images.
Therefore $\mathcal A=z_*+L$ also reduces onto the entire affine special
plane.
\end{proof}

\begin{corollary}[Special fibre in lattice coordinates]\label{cor:lattice-special-fibre}
Let $p\mid q$, and use any lattice-coordinate parametrization from
Lemma~\ref{lem:coordinates}.  The reduction of the resulting affine map
\[
 \overline\Psi:\F_p^2\longrightarrow\F_p^6
\]
has image equal to the complete affine special plane described by
\eqref{eq:Uspecial}--\eqref{eq:Vspecial}.
\end{corollary}

\begin{proof}
The parametrization in Lemma~\ref{lem:coordinates} is obtained from a basis
of the full homogeneous direction lattice $L$ and a base point $z_*$ of the
full affine lattice $z_*+L$.  Consequently reduction of its two standard coordinate
vectors identifies $\F_p^2$ with $L/pL$, and its affine image is exactly the
reduction of $z_*+L$ followed by the exact coordinate lift.  The remaining
integrality index is $\delta\mid508$, which is prime to $p\mid q$; this is
already incorporated in the isomorphism
$L/pL\simeq L_0/pL_0$ in Lemma~\ref{lem:specialfibre}.  That lemma therefore
identifies the image of $\overline\Psi$ with the entire special affine
plane.
\end{proof}

\subsection{A uniform singular-series lower bound}

For $p\mid W$, Lemma~\ref{lem:witness} supplies one simultaneous unit point
in $\F_p^2$, and therefore
\begin{equation}
 \beta_p\ge p^{-2}\left(\frac p{p-1}\right)^6>0.
 \label{eq:fixedlocal}
\end{equation}
Now let $p\nmid W$; in particular $p\ge13$.  If $p\mid q$, Corollary~\ref{cor:lattice-special-fibre} identifies the
lattice-coordinate reduction with the complete special affine plane.  If
$p\nmid q$, then $p\nmid\Delta=q^2\delta$ and
Lemma~\ref{lem:sublatticereduction} identifies the lattice-coordinate fibre
with the complete $(r,s)$-plane.  In the denominator-prime case nonvanishing
follows from \eqref{eq:qconditions}; away from the denominator it follows from
the parameter sieve and Lemma~\ref{lem:zeroforms}.  Thus every coordinate is
a nonzero affine form on $\F_p^2$.  Each has at most $p$ zeros, so at least
$p^2-6p>0$ points avoid all six zero sets.  This proves admissibility and gives
\begin{equation}
 \beta_p\ge a_p:=(1-6/p)(1-1/p)^{-6}\qquad(p\nmid W).
 \label{eq:localbound}
\end{equation}
Furthermore,
\[
 \log a_p=-\frac{15}{p^2}+O(p^{-3}),
\]
so
\[
 \sum_{p\nmid W}|\log a_p|<\infty,
 \qquad \prod_{p\nmid W}a_p>0.
\]
The preceding positivity at every prime, together with finite complexity,
shows that each retained system is admissible in the sense of
Bienvenu's Definition~1.1 \cite[Definition~1.1]{Bienvenu}.  By
Bienvenu's Lemma~1.2, its singular series $\prod_p\beta_p$ converges and
is nonzero \cite[Lemma~1.2]{Bienvenu}.  For every $X$, the finite partial
product over $p\le X$ is bounded below by the corresponding product of
\eqref{eq:fixedlocal} for $p\mid W$ and \eqref{eq:localbound} for
$p\nmid W$.  Passing to the limit gives an absolute constant
\begin{equation}
 c_*:=
 \prod_{p\mid W}p^{-2}\left(\frac p{p-1}\right)^6
 \prod_{p\nmid W}a_p>0
 \label{eq:cstar}
\end{equation}
such that
\begin{equation}
 \prod_p\beta_p\ge c_*
 \label{eq:SSlower}
\end{equation}
uniformly over all retained planes.

\section{Prime points on the planes}\label{sec:primes}

For an affine form $\psi$, write $\dot\psi$ for its linear part.  If the
ambient dimension is $d_0$, let $e_1,\dots,e_{d_0}$ be the standard basis of
$\Z^{d_0}$.  Following Bienvenu, define
\[
 \|\Psi\|_{N,B}=
 \frac1{(\log N)^B}
 \left(
 \sum_{i=1}^t\frac{|\psi_i(0)|}{N}
 +\sum_{i=1}^t\sum_{j=1}^{d_0}|\dot\psi_i(e_j)|
 \right).
\]

\begin{proposition}[Bienvenu's prime-supported theorem, uniform form]\label{prop:bienvenu}
Fix positive integers $d_0,t$ and positive constants $A_0,B_0,L_0,c_V$.
For every $\epsilon>0$, there is a threshold
\[
 N_0=N_0(d_0,t,A_0,B_0,L_0,c_V,\epsilon)
\]
with the following property.  Whenever $N\ge N_0$,
$\Psi=(\psi_1,\ldots,\psi_t):\Z^{d_0}\to\Z^t$ is admissible,
\[
 \|\Psi\|_{N,B_0}\le L_0,
\]
and $\mathcal K\subset[-N,N]^{d_0}$ is convex with
\[
 \Vol(\mathcal K)\ge c_VN^{d_0}(\log N)^{-A_0},
 \qquad
 \Psi(\mathcal K)\subset[N^{9/10},\infty)^t,
\]
one has, for $\Lambda'(m)=\1_{m\in\Pp}\log m$,
\begin{equation}
 \left|
 \frac{
  \displaystyle\sum_{n\in\mathcal K\cap\Z^{d_0}}
       \prod_{i=1}^t\Lambda'(\psi_i(n))
 }{
  \displaystyle\beta_\infty\prod_p\beta_p
 }-1\right|\le\epsilon,
 \label{eq:BienvenuUniform}
\end{equation}
where
\[
 \beta_\infty
 =\Vol\bigl(\mathcal K\cap\Psi^{-1}(\R_{>0}^t)\bigr)
 =\Vol(\mathcal K).
\]
\end{proposition}

\begin{proof}
For $N$ sufficiently large in terms of $c_V$, the inequality
$c_V\log N\ge1$ holds, and therefore
\[
 \Vol(\mathcal K)
 \ge N^{d_0}(\log N)^{-(A_0+1)}.
\]
Bienvenu's Proposition~2.1, applied in ambient dimension $d_0$ with
\[
 A=A_0+1,\qquad B=B_0,\qquad L=L_0,
\]
gives the prime-supported relative asymptotic
\[
 \sum_{n\in\mathcal K\cap\Z^{d_0}}
 \prod_{i=1}^t\Lambda'(\psi_i(n))
 =\beta_\infty\prod_p\beta_p
  \bigl(1+o_{d_0,t,A_0+1,B_0,L_0}(1)\bigr)
\]
uniformly over all admissible systems and convex bodies satisfying those
fixed bounds \cite[Proposition~2.1]{Bienvenu}.  Written with explicit
quantifiers, this uniform $o(1)$ says that for every $\epsilon>0$ there is a
single threshold, depending only on the displayed fixed parameters and
$\epsilon$, beyond which the absolute value of every individual relative
error is at most $\epsilon$.  Enlarging that threshold to include
$c_V\log N\ge1$ proves \eqref{eq:BienvenuUniform}.  Finally, the lower-value
hypothesis makes every form positive on $\mathcal K$, so Bienvenu's
archimedean factor is exactly $\Vol(\mathcal K)$.
\end{proof}

\begin{proposition}[Uniform prime count on one plane]\label{prop:perplane}
Fix $\tau>0$.  There are $T_0(\tau)$ and $c_3(\tau)>0$ such that the following
holds.  Suppose $T\ge T_0(\tau)$,
\[
 q,Z\le(\log T)^\tau,
\]
$q$ is a retained denominator, and $(A,B,C)$ is a retained parameter triple
over $q$.  Then the corresponding plane contains at least
\begin{equation}
 c_3(\tau)\frac{T^2}{q^2(\log T)^6}
 \label{eq:perplane}
\end{equation}
lattice points at which $x_1,\ldots,x_6$ are pairwise distinct primes.
\end{proposition}

\begin{proof}
Because $q$ is polylogarithmic in $T$, one has $T\ge C_0q^2\delta$ for all
sufficiently large $T$, uniformly in $\delta\mid508$.
Lemma~\ref{lem:coordinates} therefore applies.  Set
\begin{equation}
 N=\lceil C_1T\rceil,
 \qquad A_{\rm vol}=2\tau+2,
 \qquad B_0=4\tau+2,
 \label{eq:Bienvenuparameters}
\end{equation}
where $C_1$ is enlarged once, if necessary, so that
$\mathcal K\subset[-N,N]^2$.

We now check the hypotheses of Bienvenu's Definition~1.1, Lemma~1.2, and
Proposition~2.1 \cite[Definition~1.1, Lemma~1.2, Proposition~2.1]{Bienvenu}.
First, \eqref{eq:Kvolume} and $q\le(\log T)^\tau$ give
\begin{equation}
 \Vol(\mathcal K)
 \gg\frac{T^2}{q^2}
 \ge c_V(\tau)N^2(\log N)^{-A_{\rm vol}}
 \label{eq:Bienvenuvolume}
\end{equation}
for all sufficiently large $T$.  Second, by \eqref{eq:coeffbound},
\begin{align}
 &\sum_{i=1}^6\frac{|\psi_i(0)|}{N}
 +\sum_{i=1}^6\sum_{j=1}^2|\dot\psi_i(e_j)|\notag\\
 &\qquad\ll Z^2q^2\left(1+\frac1T\right)
 \ll_\tau(\log N)^{4\tau}.
 \label{eq:Bienvenusize}
\end{align}
After division by $(\log N)^{B_0}$ this is at most a fixed $L_0(\tau)$.
Third, Lemma~\ref{lem:arch} gives
\[
 \min_i\psi_i(n)\gg T\qquad(n\in\mathcal K).
\]
Since $N\asymp T$, this implies
\begin{equation}
 \Psi(\mathcal K)\subset[N^{9/10},\infty)^6
 \label{eq:Bienvenupositive}
\end{equation}
for all sufficiently large $T$.

The exact fixed parameters in Bienvenu's Proposition~2.1 are
\begin{equation}
 \begin{aligned}
 d_0&=2,\qquad t=6,\qquad A=A_{\rm vol}+1=2\tau+3,\\
 B&=B_0=4\tau+2,\qquad L=L_0(\tau).
 \end{aligned}
 \label{eq:Bienvenuapplicationdata}
\end{equation}
The black-box hypotheses are therefore accounted for as follows.
\begin{center}
\begingroup
\small
\renewcommand{\arraystretch}{1.18}
\begin{tabular}{>{\raggedright\arraybackslash}p{0.34\textwidth}>{\raggedright\arraybackslash}p{0.55\textwidth}}
\toprule
Bienvenu hypothesis&Present check\\
\midrule
$d_0=2$, $t=6$&Fixed throughout this section.\\
$\mathcal K\subset[-N,N]^2$&Lemma~\ref{lem:coordinates} and the choice of $N$.\\
$\Vol(\mathcal K)\gg N^2(\log N)^{-A_{\rm vol}}$&Equation~\eqref{eq:Bienvenuvolume}.\\
$\|\Psi\|_{N,B_0}\le L_0$&Equation~\eqref{eq:Bienvenusize}, with $B_0=4\tau+2$.\\
$\Psi(\mathcal K)\subset[N^{9/10},\infty)^6$&Equation~\eqref{eq:Bienvenupositive}.\\
Finite complexity&Lemma~\ref{lem:complexity}, preserved by the invertible lattice reparametrization.\\
$\beta_p>0$ for every $p$&The cases $p\mid W$, $p\mid q$, and $p\nmid Wq$ in Section~\ref{sec:local}.\\
$\prod_p\beta_p\ge c_*$&Equation~\eqref{eq:SSlower}, with convergence from Bienvenu's Lemma~1.2.\\
\bottomrule
\end{tabular}
\endgroup
\end{center}

Indeed, the lattice-coordinate map has invertible linear part over $\R$, so
two pulled-back linear parts are proportional only if the original ones are.
Thus finite complexity is preserved.  Apply Proposition~\ref{prop:bienvenu}
with
\[
 \begin{aligned}
 d_0&=2, & t&=6, & A_0&=A_{\rm vol},\\
 B_0&=4\tau+2, & L_0&=L_0(\tau), & c_V&=c_V(\tau),
 \qquad \epsilon=\tfrac12.
 \end{aligned}
\]
Equivalently, Bienvenu's source proposition is used with the fixed tuple in
\eqref{eq:Bienvenuapplicationdata}.  Increase $T_0(\tau)$ so that
$N=\lceil C_1T\rceil$ exceeds the single threshold supplied by
Proposition~\ref{prop:bienvenu} for these parameters.  Then every retained
plane satisfies the same relative-error bound $1/2$, and therefore
\begin{equation}
 \sum_{n\in\mathcal K\cap\Z^2}
 \prod_{i=1}^6\Lambda'(\psi_i(n))
 \ge\frac{c_*}{2}\Vol(\mathcal K)
 \gg_\tau\frac{T^2}{q^2}.
 \label{eq:weighted}
\end{equation}

Every prime value is $O(Z^2T)$, and therefore its logarithm is
$O_\tau(\log T)$.  Dividing \eqref{eq:weighted} by the maximum possible
product of the six logarithmic weights gives \eqref{eq:perplane} before
coordinate collisions are removed.

For $i\ne j$, the equation $\psi_i(n)=\psi_j(n)$ is a genuine affine line,
because the two linear parts are nonproportional.  Distinct integer points on
a line have Euclidean separation at least $1$, while its intersection with
$[-C_1T,C_1T]^2$ has length $O(T)$.  Thus every collision line contains
$O(T)$ integer points, uniformly in its coefficients.  Since
\[
 \frac{T}{T^2/(q^2(\log T)^6)}
 =\frac{q^2(\log T)^6}{T}\longrightarrow0
\]
uniformly for $q\le(\log T)^\tau$, deleting all fifteen collision lines
preserves a fixed proportion of the lower bound.
\end{proof}

\medskip
\noindent\textbf{Intersection accounting.}
Let two retained planes have parameter triples $\theta\ne\theta'$.  By
Lemma~\ref{lem:distinctplanes}, their skew matrices are distinct.  The matrix
$A_\theta-A_{\theta'}$ is a nonzero $3\times3$ skew-symmetric matrix, and
hence has rank two.  Thus the corresponding homogeneous graph spaces meet in
the line
\[
 \{(U,A_\theta U):U\in\ker(A_\theta-A_{\theta'})\}
\]
through the origin.  The retained affine planes are their intersections with
$\zeta-y_5=1$.  Since a line through the origin cannot be contained in that
affine hyperplane, two distinct retained planes meet in at most one point.

More generally, let $\mathcal P_{\rm pl}$ be a finite family of retained planes and let
$m(z)$ be the number of its members containing $z$.  Then
\begin{equation}
 \sum_z(m(z)-1)
 \le\sum_z\binom{m(z)}2
 \le\binom{\#\mathcal P_{\rm pl}}2,
 \label{eq:intersectionaccounting}
\end{equation}
because a point on $m(z)$ planes determines $\binom{m(z)}2$ unordered pairs
of planes and each pair has at most one common point.

\section{Summing the ruling and counting cores}\label{sec:summing}

Choose absolute constants $0<c_h<C_h$ for which Lemma~\ref{lem:arch} gives
\begin{equation}
 c_hZ^3T^2\le508x_5x_6\le C_hZ^3T^2
 \label{eq:heightbounds}
\end{equation}
throughout the positive region, and put
\[
 c_0=\frac{c_h}{C_h}\in(0,1).
\]

\begin{lemma}[Uniform polylogarithmic thresholds]\label{lem:uniformthresholds}
Fix $\kappa>0$.  Suppose
\[
 1\le Q,Z\le(\log Y)^\kappa,
 \qquad q\in[Q,2Q].
\]
Put
\begin{equation}
 T=\left(\frac{Y}{C_hZ^3}\right)^{1/2}.
 \label{eq:Tchoice}
\end{equation}
As $Y\to\infty$, the following statements hold uniformly in all such
$Q,Z,q$:
\begin{align}
 \log T&\asymp_\kappa\log Y,\label{eq:logcompare}\\
 q,Z&\le(\log T)^{\kappa+1},\label{eq:polytransfer}\\
 \frac{T}{q^2}&\longrightarrow\infty,
 &\frac{T}{q^2(\log T)^6}&\longrightarrow\infty,
 \label{eq:Tthresholds}\\
 \frac{Q^8Z^8}{T^2Q^2Z^4/(\log T)^6}&\longrightarrow0.
 \label{eq:overlapthreshold}
\end{align}
In addition, $ZT\to\infty$ uniformly.
\end{lemma}

\begin{proof}
From \eqref{eq:Tchoice},
\[
 \log T=\frac12\log Y-\frac32\log Z-\frac12\log C_h.
\]
Since $\log Z\le\kappa\log\log Y$, we have, for all sufficiently large $Y$,
\[
 \frac13\log Y\le\log T\le\log Y,
\]
which proves \eqref{eq:logcompare}.  It also gives
\[
 q\le2(\log Y)^\kappa
 \le2\cdot3^\kappa(\log T)^\kappa
 \le(\log T)^{\kappa+1}
\]
for large $Y$, and the same argument applies to $Z$; this proves
\eqref{eq:polytransfer}.

Uniformly in the stated range,
\[
 \frac{T}{q^2}
 \gg_\kappa\frac{Y^{1/2}}{Z^{3/2}(\log Y)^{2\kappa}}
 \ge\frac{Y^{1/2}}{(\log Y)^{7\kappa/2}},
\]
which tends to infinity.  Division by an additional factor
$(\log T)^6\ll(\log Y)^6$ still leaves a quantity tending to infinity,
proving \eqref{eq:Tthresholds}.  Finally,
\begin{align*}
 \frac{Q^8Z^8}{T^2Q^2Z^4/(\log T)^6}
 &=\frac{Q^6Z^4(\log T)^6}{T^2}\\
 &=\frac{C_hQ^6Z^7(\log T)^6}{Y}\\
 &\ll_\kappa\frac{(\log Y)^{13\kappa+6}}{Y}
 \longrightarrow0,
\end{align*}
which is \eqref{eq:overlapthreshold}.  The lower bound for $T$ above also
implies $ZT\to\infty$.
\end{proof}

We now prove Theorem~\ref{thm:cores}.  Fix $\kappa>0$, and choose
\[
 Q_\kappa\ge\max\{Q_0,q_0,3\},
 \qquad Z_\kappa\ge\max\{Z_0,100\}.
\]
For given $Y,Q,Z$ in the ranges of Theorem~\ref{thm:cores}, define $T$ by
\eqref{eq:Tchoice}.  By Lemma~\ref{lem:uniformthresholds}, after choosing
$Y_\kappa$ sufficiently large, all of the following hold simultaneously and
uniformly: Proposition~\ref{prop:perplane} applies with $\tau=\kappa+1$; the
lattice requirement $T\ge C_0\Delta$ holds; coordinate collisions and plane
overlaps are negligible; every prime value exceeds $127$; and $ZT$ is large
enough for the discrepancy estimate below.

Lemma~\ref{lem:denominators} supplies $\gg Q$ retained denominators
$q\in[Q,2Q]$.  Proposition~\ref{prop:parametersieve} supplies
$\gg q^3Z^4$ retained planes over each such $q$, while
Proposition~\ref{prop:perplane} supplies
$\gg_\kappa T^2/(q^2(\log T)^6)$ pairwise-distinct prime points on each
plane.  The total incidence count is therefore
\begin{equation}
 I(T;Q,Z)\gg_\kappa
 \frac{T^2Z^4}{(\log T)^6}\sum_q q
 \gg_\kappa\frac{T^2Q^2Z^4}{(\log T)^6}.
 \label{eq:incidence}
\end{equation}

Let $m(z)$ denote the number of retained planes contributing the prime point
$z$.  Then
\begin{equation}
 \#\{z:m(z)\ge1\}
 =\sum_zm(z)-\sum_z(m(z)-1)
 =I(T;Q,Z)-\sum_z(m(z)-1).
 \label{eq:distinctfromincidence}
\end{equation}
Let $\mathcal P_{\rm pl}$ be the family of retained planes used in
\eqref{eq:incidence}.  For a fixed $q$, the total number of candidate
parameter triples is at most
\[
 O\left(\frac{qZ^2}{W}+1\right)
 O\left(\frac{qZ}{W}+1\right)^2.
\]
Therefore, since $Q,Z\ge1$,
\begin{equation}
 \#\mathcal P_{\rm pl}
 \le\sum_{Q\le q\le2Q}
 O\left(\left(\frac{qZ^2}{W}+1\right)
          \left(\frac{qZ}{W}+1\right)^2\right)
 =O_W(Q^4Z^4).
 \label{eq:planefamilyupper}
\end{equation}
Equation~\eqref{eq:intersectionaccounting} consequently gives
\[
 \sum_z(m(z)-1)\le\binom{\#\mathcal P_{\rm pl}}{2}=O_W(Q^8Z^8).
\]
Moreover,
\begin{equation}
 \frac{Q^8Z^8}{T^2Q^2Z^4/(\log T)^6}
 =\frac{Q^6Z^4(\log T)^6}{T^2}=o(1)
 \label{eq:overlapratioexplicit}
\end{equation}
by \eqref{eq:overlapthreshold}.  Hence the second term in
\eqref{eq:distinctfromincidence} is $o(I(T;Q,Z))$, and
\eqref{eq:incidence} gives the same order of magnitude of distinct ordered
prime sextuples.

For each such sextuple, define
\[
 u=15p_1p_2,\qquad v=31p_3p_4,
 \qquad h=508p_5p_6.
\]
Equation \eqref{eq:product} gives $u+v=h$.  The six primes are pairwise
distinct, exceed $127$, and avoid all fixed prime divisors, so $(u,v)=1$.
By \eqref{eq:heightbounds} and \eqref{eq:Tchoice},
\[
 c_0Y\le h\le Y.
\]
Given one ordered core $(u,v)$, the perimeter $h=u+v$ is also fixed.  Since
all six variable primes exceed $127$ and are pairwise distinct, none divides
$15\cdot31\cdot508$, and unique factorization gives the three unordered
prime pairs exactly:
\[
 \{p_1,p_2\}\ \text{from }u/15,\qquad
 \{p_3,p_4\}\ \text{from }v/31,\qquad
 \{p_5,p_6\}\ \text{from }h/508.
\]
The only remaining ambiguity is the order within each of these three pairs.
Thus at most $2^3=8$ ordered prime sextuples yield the same ordered core.
Since $T^2=Y/(C_hZ^3)$ and
$\log T\asymp_\kappa\log Y$, the number of distinct generated cores is at
least
\begin{equation}
 c_\kappa\frac{YQ^2Z}{(\log Y)^6}.
 \label{eq:corecount}
\end{equation}

The following final clause is not needed for the superlinear lower bound;
it records the additional structural fact that the fixed non-coprime
multiplier $3600$ is genuinely essential for the generated cores.  Since
\[
 \sigma(15)=24,\qquad\sigma(31)=32,\qquad\sigma(508)=896,
\]
the quadratic contribution is exactly
\begin{equation}
 24p_1p_2+32p_3p_4-896p_5p_6
 =-\frac{88}{5}p_3p_4-\frac{416}{5}p_5p_6.
 \label{eq:negativequad}
\end{equation}
Consequently the complete discrepancy is
\begin{align}
 \sigma(u)+\sigma(v)-\sigma(h)
 ={}&-\frac{88}{5}p_3p_4-\frac{416}{5}p_5p_6\notag\\
 &+24(p_1+p_2+1)+32(p_3+p_4+1)\notag\\
 &-896(p_5+p_6+1).
 \label{eq:discrepancy}
\end{align}
The exact derivation of \eqref{eq:negativequad} from \eqref{eq:product} is
recorded again in Appendix~\ref{app:algebra}.  Lemma~\ref{lem:arch} gives
absolute constants $c_4,C_4>0$ for which
\[
 p_3p_4,p_5p_6\ge c_4Z^3T^2,
 \qquad p_1+\cdots+p_6\le C_4Z^2T.
\]
It follows from \eqref{eq:discrepancy} that, after changing the absolute
constants if necessary,
\[
 \sigma(u)+\sigma(v)-\sigma(h)
 \le-c_4Z^3T^2+C_4Z^2T<0
\]
whenever $ZT>C_4/c_4$.  This is uniform by
Lemma~\ref{lem:uniformthresholds}.  Hence the reduced core is not itself a
solution.  On the other hand, Section~\ref{sec:friendly} proves
\[
 \sigma(3600u)+\sigma(3600v)=\sigma(3600h).
\]

For completeness, the single uniform threshold $Y_\kappa$ is chosen only
after every fixed constant above, and large enough that for all admissible
$Q,Z$ and every $q\in[Q,2Q]$ one has simultaneously
\begin{gather}
 T\ge T_0(\kappa+1),\qquad
 T\ge C_0\cdot508(2Q)^2,\qquad
 ZT\ge \max\{128/c,\ C_4/c_4\},\label{eq:thresholdaudit1}\\
 \frac{(2Q)^2(\log T)^6}{T}\le\varepsilon_{\rm coll},
 \qquad
 \frac{Q^6Z^4(\log T)^6}{T^2}\le\varepsilon_{\rm ov}.
 \label{eq:thresholdaudit2}
\end{gather}
Here $c$ is the lower constant in Lemma~\ref{lem:arch}, while the fixed
positive constants $\varepsilon_{\rm coll}$ and $\varepsilon_{\rm ov}$ are
chosen small enough for the collision and overlap deletions to retain, say,
half of their respective main terms.  The first inequality in
\eqref{eq:thresholdaudit1} invokes the per-plane theorem; the second follows
from $\Delta=q^2\delta\le508(2Q)^2$ and invokes the lattice-coordinate lemma; the
third ensures both $p_i>127$ and the strict negativity of the core
discrepancy.  Lemma~\ref{lem:uniformthresholds} proves that one
$Y_\kappa$ enforces all these inequalities uniformly.  This completes the
proof of Theorem~\ref{thm:cores}.

\section{Multiplier amplification and proof of the main theorem}

For one generated core, put
\[
 \mathfrak P=\rad(3600uvh).
\]
Besides the fixed primes $2,3,5,31,127$, the prime divisors of $\mathfrak P$ are
exactly $p_1,\dots,p_6$.  Consequently there is an absolute $c_5>0$ such
that
\begin{equation}
 \frac{\varphi(\mathfrak P)}{\mathfrak P}\ge c_5,
 \qquad 2^{\omega(\mathfrak P)}\le2^{11}.\label{eq:uniformphi}
\end{equation}
Every $g$ with $(g,\mathfrak P)=1$ gives another ordered solution
\[
 (3600gu,3600gv),
\]
because $g$ is coprime to each of $3600u,3600v,3600h$, so the common factor
$\sigma(g)$ may be extracted from all three divisor sums.  Inclusion--
exclusion and \eqref{eq:uniformphi} give absolute constants $c_6,X_0>0$
such that
\begin{equation}
 \#\{g\le X:(g,\mathfrak P)=1\}
 =\frac{\varphi(\mathfrak P)}{\mathfrak P}X
  +O(2^{\omega(\mathfrak P)})
 \ge c_6X\qquad(X\ge X_0).\label{eq:multipliers}
\end{equation}
Multiplier families belonging to different ordered cores are disjoint.  Indeed, suppose
\[
 (3600gu,3600gv)=(3600g'u',3600g'v'),
\]
where $(u,v)$ and $(u',v')$ are coprime cores.  Then
\[
 3600g=\gcd(3600gu,3600gv)
 =\gcd(3600g'u',3600g'v')=3600g'.
\]
Thus $g=g'$, and division by $3600g$ gives $(u,v)=(u',v')$.

Fix $\kappa>0$.  Choose $\chi>1/c_0$, and take geometrically spaced scales
$Y_j=\chi^j$ in the range
\[
 x^{1/2}\le Y_j\le\frac{x}{3600X_0}.
\]
The shells $[c_0Y_j,Y_j]$ are disjoint and their number is
$\asymp\log x$.  For each shell put
\[
 Q_j=Z_j=\left\lfloor\frac13(\log Y_j)^\kappa\right\rfloor.
\]
For sufficiently large $x$, these are above all lower thresholds in
Theorem~\ref{thm:cores}; moreover
\[
 2Q_j\le(\log Y_j)^\kappa,
 \qquad Z_j\le(\log Y_j)^\kappa,
\]
so every selected denominator $q\in[Q_j,2Q_j]$ remains in the required
polylogarithmic coefficient range.  By \eqref{eq:corecount}, the $j$th shell
contains
\[
 \gg_\kappa Y_j(\log Y_j)^{3\kappa-6}
\]
cores.  For a core of perimeter $h\le Y_j$, take
$X=x/(3600h)$.  The upper restriction on $Y_j$ ensures $X\ge X_0$, and
\eqref{eq:multipliers} supplies $\gg x/Y_j$ admissible multipliers $g$ for
which $3600gh\le x$.  Hence each shell contributes
\[
 \gg_\kappa x(\log Y_j)^{3\kappa-6}
\]
ordered solutions.  Since $\log Y_j\asymp\log x$ throughout the chosen
range, summing the disjoint shells gives
\begin{equation}
 S(x)\gg_\kappa x(\log x)^{3\kappa-5}.\label{eq:masterlower}
\end{equation}
Given $R>0$, choose a fixed $\kappa>(R+5)/3$.  Then
\[
 \frac{S(x)}{x(\log x)^R}
 \gg_\kappa(\log x)^{3\kappa-5-R}\longrightarrow+\infty.
\]
This proves Theorem~\ref{thm:main}.

\section{Conclusion}

We have proved that the counting function in Erd\H{o}s Problem 1061 satisfies
\[
 S(x)\gg_\kappa x(\log x)^{3\kappa-5}
\]
for every fixed $\kappa>0$.  Since $\kappa$ is arbitrary, this gives
\[
 \lim_{x\to\infty}\frac{S(x)}{x(\log x)^R}=+\infty
 \qquad(R>0),
\]
and hence resolves the question by ruling out every linear asymptotic
$S(x)\sim cx$.

The construction has two structural ingredients.  First, a fixed equal-
abundancy triple reduces the divisor-sum equation to a prime-point problem on
a split quadric.  Second, a three-parameter rational ruling of that quadric,
a lattice-index computation, a codimension-two parameter sieve, and Bienvenu's
uniform theorem for prime-supported affine systems produce enough prime
points that coprime multiplier amplification exceeds all fixed logarithmic
scales.  The multiplier $3600$ is not incidental: the reduced cores produced
above are not themselves solutions, but become solutions after this fixed
non-coprime activation.

\appendix

\section{Exact algebraic identities}\label{app:algebra}

This appendix records the finite algebra used in the proof.

\subsection{The quadric substitution and the expanded sixth form}

Solving \eqref{eq:linear} after writing
$x_i=y_i+y_5$ for $1\le i\le4$ and $x_5=y_5$ gives
\begin{equation}
 508x_6=15(y_1+y_2)+31(y_3+y_4)-416y_5-462.
 \label{eq:appx6y}
\end{equation}
A direct expansion then yields the exact identity
\begin{align}
 &15x_1x_2+31x_3x_4-508x_5x_6\notag\\
 &\qquad=15y_1y_2+31y_3y_4+462y_5(y_5+1),
 \label{eq:appquadric}
\end{align}
which proves the passage to \eqref{eq:split} after putting $\zeta=y_5+1$.

For the ruling coordinates \eqref{eq:x1}--\eqref{eq:x6}, define
\begin{align*}
 E&=6930\beta+14322\gamma+46,\\
 E_r&=-31\alpha-6930\beta^2-14322\beta\gamma+416\beta+1,\\
 E_s&=15\alpha-6930\beta\gamma-14322\gamma^2+416\gamma+1.
\end{align*}
Then the fully expanded sixth form is
\begin{equation}
 508x_6=-E+E_r r+E_s s.\label{eq:expandedx6}
\end{equation}
Substituting \eqref{eq:x1}--\eqref{eq:x5} and
\eqref{eq:expandedx6} gives identically
\[
 15x_1x_2+31x_3x_4=508x_5x_6
\]
and
\[
 15(x_1+x_2)+31(x_3+x_4)-508(x_5+x_6)=462.
\]

\subsection{All fifteen determinant factorizations}

With $\ell_i=(\partial_rx_i,\partial_sx_i)$, the exact determinants are
\begin{center}
\begingroup
\footnotesize
\renewcommand{\arraystretch}{1.18}
\setlength{\tabcolsep}{4pt}
\begin{tabular}{c|l}
\toprule
pair&$\det(\ell_i,\ell_j)$\\
\midrule
$1,2$&$-\bigl(15\alpha\beta-\alpha+462\beta\gamma+\gamma\bigr)/15$\\
$1,3$&$-\bigl(15\beta+31\gamma-1\bigr)/465$\\
$1,4$&$-\gamma(15\alpha+462\gamma+1)/15$\\
$1,5$&$-\gamma/15$\\
$1,6$&$-(15\alpha+462\gamma+1)(15\beta+31\gamma-1)/7620$\\
$2,3$&$\beta(31\alpha-462\beta-1)/31$\\
$2,4$&$\alpha(\alpha+\beta-\gamma)$\\
$2,5$&$\alpha\beta$\\
$2,6$&$(\alpha+\beta-\gamma)(31\alpha-462\beta-1)/508$\\
$3,4$&$-(31\alpha\gamma-\alpha-462\beta\gamma-\beta)/31$\\
$3,5$&$\beta/31$\\
$3,6$&$-(31\alpha-462\beta-1)(15\beta+31\gamma-1)/15748$\\
$4,5$&$\alpha\gamma$\\
$4,6$&$-(\alpha+\beta-\gamma)(15\alpha+462\gamma+1)/508$\\
$5,6$&$-(15\alpha\beta+31\alpha\gamma+\beta-\gamma)/508$\\
\bottomrule
\end{tabular}
\endgroup
\end{center}
These identities justify the numerator table in
Lemma~\ref{lem:complexity} without suppressing any parameter-dependent
factor.

\subsection{The three zero-form identities}

The first two criteria in Lemma~\ref{lem:zeroforms} follow from
\begin{equation}
 x_2\big|_{\beta=-1/462}=\alpha s,
 \qquad
 x_4\big|_{\gamma=-1/462}=-\alpha r.
 \label{eq:appzero24}
\end{equation}
For $x_6$, the quantities in \eqref{eq:expandedx6} satisfy
\begin{align}
 508x_6(0,0)&=-E,\label{eq:appzero60}\\
 E_r+\beta E&=-(31\alpha-462\beta-1),\label{eq:appzero6r}\\
 E_s+\gamma E&=15\alpha+462\gamma+1.
 \label{eq:appzero6s}
\end{align}
If the right sides of \eqref{eq:appzero6r} and
\eqref{eq:appzero6s} vanish, then
\[
 \beta=\frac{31\alpha-1}{462},
 \qquad
 \gamma=-\frac{15\alpha+1}{462},
\]
which gives $E=0$.  Equations \eqref{eq:appzero60}--\eqref{eq:appzero6s}
then prove both directions of the $x_6$ criterion.

\subsection{The core discrepancy}

Using \eqref{eq:product},
\begin{align*}
 &24p_1p_2+32p_3p_4-896p_5p_6\\
 &\quad=\frac{24}{15}(508p_5p_6-31p_3p_4)
       +32p_3p_4-896p_5p_6\\
 &\quad=-\frac{88}{5}p_3p_4-\frac{416}{5}p_5p_6,
\end{align*}
which is the quadratic identity used in \eqref{eq:discrepancy}.

\end{document}